\documentclass[11pt,english]{smfart}

\usepackage{eucal,dsfont,mathptmx,newtxtext,newtxmath}

\usepackage[text={6in,9in},centering]{geometry} 
\usepackage[dvipsnames]{xcolor}   
\usepackage{xparse}
\usepackage{xr-hyper}
\usepackage[linktocpage=true,colorlinks=true,hyperindex,citecolor=BrickRed,linkcolor=blue]{hyperref}  

\newcommand{\inv}{^{\raisebox{.2ex}{$\scriptscriptstyle-1$}}}   

\DeclareMathOperator{\id}{\mathcal{I}(\textit{M})}
\DeclareMathOperator{\st}{\mathcal{I}^{++}(\textit{M})}

\DeclareMathOperator{\ir}{\mathcal{I}^+(\textit{M})}

\theoremstyle{plain}
\newtheorem{theorem}{Theorem}[section]
\newtheorem{proposition}[theorem]{Proposition}
\newtheorem{lemma}[theorem]{Lemma}
\newtheorem{corollary}[theorem]{Corollary}
\theoremstyle{definition}

\theoremstyle{remark}

\theoremstyle{definition}

\numberwithin{equation}{section}

\renewcommand{\labelenumi}{\upshape(\arabic{enumi})}

\makeatletter
\def\msection{\@startsection{section} 
	{1} 
	{0pt} 
	{-1ex plus -.1ex minus -0.9ex} 
	{-.9ex plus -.2ex} 
	{\bfseries} 
}
\def\msubsection{\@startsection{subsection} 
	{2} 
	{0pt} 
	{-1ex plus -.1ex minus -0.2ex} 
	{-.9ex plus -.2ex} 
	{\normalfont} 
} 
\makeatother 

\begin{document}
\author{Amartya Goswami}
\address{
[1] Department of Mathematics and Applied Mathematics, University of Johannesburg, P.O. Box 524, Auckland Park 2006, South Africa.
[2] National Institute for Theoretical and Computational Sciences (NITheCS), South Africa.}
\email{agoswami@uj.ac.za}

\title{Some results on  irreducible ideals  of monoids}

\subjclass{20M12, 20M14}


\keywords{Monoid, irreducible ideal, localization.}

\begin{abstract}
The purpose of this note is to study some algebraic properties of irreducible ideals of monoids. We establish relations between  irreducible, prime, and semiprime ideals. We explore some properties of irreducible ideals in local, Noetherian, and Laskerian monoids. 
\end{abstract}
\maketitle

\msection{Introduction and preliminaries}
 
A comprehensive ideal theory for monoids was originally introduced in \cite{Aub62}, and has since been extended through numerous studies (see \cite{FGKT17, GK22, GGT21, GR20, GZ20, GHL07, Ger96, H-K98}). In the context of algebraic geometry, the recent advancements in logarithmic algebraic geometry and toric varieties have shown interest in the study of ideals associated with monoids, as indicated in \cite{Ogu18, Kat94, Put88, Ren05, CHWW15}.

In this paper, we study a specific type of ideals of monoids, namely, irreducible ideals. We extend certain algebraic properties observed in strongly irreducible ideals of rings or semirings to irreducible ideals of monoids. 

By a monoid $M $, we shall mean a system $(M , \cdot,1)$ such that $\cdot$ is an associative, commutative  (multiplicative) binary operation on $M $ and $1$ is an element of $M $ such that $x\cdot 1=x$ for all $x\in M $. We shall write $xy$ for $x\cdot y$. A monoid is called \emph{pointed} if there exists an element $0\in M$ such that $x\cdot 0=0$ for all $x\in M.$ In this paper, all our monoids are pointed. 

If $S$ and $T$ are subsets of a monoid $M$, then by the \emph{set product} $ST$ of $S$ and $T$, we shall mean $ST:=\left\{st\mid s\in S, t\in T\right\}.$ If $S=\{s\}$ we write $ST$ as $sT$, and similarly for $T=\{t\}.$ Thus
\[ST=\bigcup\left\{ St\mid t\in T\right\}=\bigcup \left\{sT\mid s\in S\right\}.\]

Recall that an \emph{ideal}  of a monoid $M $ is a nonempty subset $I$ of $M$ such that whenever $i\in I$ and
$m\in M $ implies $im \in I$. An ideal $I$ of $M$ is called \emph{proper} if $I\neq M $. By $\id$, we shall denote the set of all ideals of $M$. If $S$ is a nonempty subset of a monoid $M$, then $\langle S\rangle$ will denote the smallest ideal generated by $S$. If $S=\{s\},$ then we shall write $\langle s \rangle$ for $\langle \{s\} \rangle$.

An abstraction of the ideal theory of a commutative
ring is the theory of
$x$-ideals given in \cite{Aub62}.
If $S$ 
is a commutative semigroup, then we say  $S$ is endowed with \emph{$x$-system} if for every subset $A$ of $S$ there corresponds a subset $A_x$
of $S$ such that for any subsets $A$ and $B$ fo $S$, the following conditions are satisfied:
\begin{itemize}
\item[]$\bullet$ $A\subseteq A_x$,

\item[]$\bullet$ $A\subseteq B_x\Rightarrow A_x\subseteq B_x$,

\item[]$\bullet$ $AB_x\subseteq B_x\cap (AB)_x$.
\end{itemize}
If we take
$A_x
= SA \cup A$, then
we have the above-mentioned notion of the  ideal generated by $A$, which in \cite{Aub62} is
called an
\emph{$x$-ideal}. In this paper, we shall assume all our ideals are $x$-ideals.

Let us record some elementary definitions about ideals of monoids. If $I$ and $J$ are ideals of a monoid $M $, then their \emph{product} is defined by \[IJ :=\left\{ij\mid i\in I, j\in J\right\},\] which is also an ideal of $M $.   Let $S$ be a nonempty subset of a monoid $M$. An \emph{ideal quotient} or a \emph{colon ideal} is defined by \[(I: S) :=\left\{ m\in M\mid mS\subseteq I\right\}.\] It is easy to check that $(I:S)$ is indeed an ideal of $M$.  A proper ideal $P$ is called \emph{prime} if $xy\in P$ implies that $x\in P$ or $y\in P$ for all $x,$ $y\in M.$ By $\mathrm{Spec}(M)$, we shall denote the set of prime ideals of $M$. A proper ideal $J$ of $M$ is called $\emph{maximal}$ if it is not properly contained in another proper ideal of $M$.  If $I$ is an ideal of $M $, the \emph{radical} of $I$ is defined by \[\sqrt{I} :=\left\{m\in M \mid m^k\in I\;\text{for some}\;k\in \mathds{Z}^+\right\}.\] An ideal $I$ is said to be a \emph{radical ideal} (or to be a \emph{semiprime}) if $\sqrt{I}=I.$  A proper ideal $I$ of $M$ is called \emph{primary} if $xy\in I$ implies $x\in I$ or $y^k\in I$ for some $k\in \mathds{Z}^+.$

An  ideal $L$ of $M $ is called \emph{irreducible} (\emph{strongly irreducible})  if for ideals $I,$ $J$ of $M$ and $L=I\cap J$ ($L\subseteq I\cap J $) implies that $L=I$ ($L\subseteq I$) or $L=J$ ($L\subseteq J$) .We shall denote the sets of all irreducible  and strongly irreducible ideals of $M$ respectively by $\ir$ and $\st$.  It is easy to see that every strongly irreducible ideal is irreducible. The converse of this also holds in monoids, which follows from the fact that every monoid is \emph{arithmetic}, that is, the set $\id$ of all ideals  of $M$ forms a distributive lattice under set inclusion as the partial order. 

To prove a property for an irreducible ideal, thanks to the equivalence between $\ir$ and $\st$, it is sufficient to check it  for  the strongly irreducible condition.

In the following lemma, we gather some properties of  ideals of monoids that will be used in sequel.

\begin{lemma}\label{prlm}
Let $M$ be monoid.
\begin{enumerate}
	
\item\label{umi}
The unique maximal ideal of $I$ is the set of non-invertible elements of $M$.
	
\item \label{psin}
If $I$ and $J$ are two ideals of $M$, then $IJ\subseteq I\cap J$.

\item\label{ijij} If $i,$ $j\in M$, then $\langle i \rangle \langle j \rangle = \langle ij \rangle$.
	
\item An ideal $P$ of $M$ is prime if and only if, for all $I,$ $J\in \id$ and $IJ\subseteq  P$ implies $I\subseteq P$ or $J\subseteq P$.

\item An ideal $I$ of $M$ is semiprime if and only if $J^2\subseteq I$ implies $J\subseteq I,$ for all $J\in \id.$
\end{enumerate}
\end{lemma}

\begin{proof}  (1)--(3) Straightforward.

(4) Let $P$ be a prime ideal of $M$ and $IJ\subseteq P$ for some $I, J\in \id.$ Let $I\nsubseteq P$. Then $i\notin P$ for some $i\in I$. However, for all $j\in J$, we have $ij\in IJ\subseteq P$. This implies $j\in P$ for all $j\in J$. Hence $J\subseteq P$. To show the converse, let $ij\in P$ for some $i, j\in M$. Then  $\langle i \rangle \langle j \rangle = \langle ij \rangle$, by (\ref{ijij}); and therefore, $i\in \langle i \rangle \subseteq P$ or $j\in \langle j \rangle \subseteq P$.

5. If $I$ is semiprime, then the  condition holds trivially. The converse follows by induction.
\end{proof}
\smallskip

\msection{Irreducible ideals}

Here is an elementwise equivalent definition of  irreducible ideals of monoids, and as mentioned before, we shall check only the strongly irreducible condition.

\begin{lemma}
\label{asi} 
An ideal $I$ of $M$ is  irreducible if and only if $\langle m \rangle \cap \langle m' \rangle \subseteq I$ implies $m\in I$ or $m'\in I$, for all $m, m'\in M$.
\end{lemma}

\begin{proof}
Let $I$ be an irreducible ideal of $M$ and $\langle m \rangle \cap \langle m' \rangle \subseteq I$, for some $m, m'\in M$. Then $m\in\langle m \rangle  \subseteq I$ or $m'\in\langle m' \rangle  \subseteq I$. Conversely, let $I$ be an ideal of $M$ with $J\cap K\subseteq I$ for some ideals $J$ and $K$ of $M$. Let $J\nsubseteq I$. This implies $j\notin I$ for some $j\in J$, and hence $\langle j\rangle \nsubseteq I.$ Therefore, for all $k\in K$, we have \[\langle j\rangle \cap \langle k\rangle \subseteq J\cap K\subseteq I,\] and by hypothesis, we obtain $k\in I$ for all $k\in K$. In other words, $K\subseteq I$.
\end{proof}

The following result shows when all proper ideals of a monoid are  irreducible. Since the proof is straightforward, we skip it. This result generalizes \cite[Lemma 3.5]{Azi08}.

\begin{lemma}
Let $M$ be a monoid. Then the following are equivalent.
\begin{enumerate}
\item Every proper ideal of $M$ is an irreducible ideal.

\item Every two ideals of $M$ are comparable.
\end{enumerate}
\end{lemma}

In the following proposition we shall see some relations of irreducible ideals with prime and semiprime ideals of a monoid. This result extend \cite[Proposition 2]{AA08}, \cite[Theorem 2.1]{Azi08}, and \cite[Proposition 7.36]{G99}.

\begin{proposition}\label{ils}
Let $M$ be a monoid. Then the following hold.
\begin{enumerate}
\item \label{siii}
Every prime ideal of $M$ is strongly irreducible, and hence, irreducible.

\item 
An ideal $P$ of  $M $ is prime if and only if it is semiprime and irreducible.
\end{enumerate}
\end{proposition}

\begin{proof} 
To obtain (1), it suffices to notice that by Lemma \ref{prlm}(\ref{psin}), the condition: $IJ\subseteq I\cap J$ holds for any two ideals $I$ and $J$ of $M$. 
For (2), let $P$ be a prime ideal. Then  $P$ is strongly irreducible, and hence, irreducible, by (\ref{siii}). To show $P$ is semiprime, it suffices to show $\sqrt{P}\subseteq P$, which follows immediately; indeed, $m\in \sqrt{P}$ implies $m^k\in P$, for some $k\in \mathds{Z}^+$, and hence $m\in P$ as $P$ is a prime ideal. For the converse, let $P$ be a strongly irreducible ideal and $IJ\subseteq P$, for some ideals $I$ and $J$ of $M$. Since
\[(I\cap J)^2\subseteq IJ\subseteq P,\]
and since $P$ is semiprime, $I\cap J\subseteq P$. Since $P$ is strongly irreducible, $I\subseteq P$ or $J\subseteq P.$
\end{proof}

By Lemma \ref{prlm}(\ref{umi}), every proper ideal of a monoid $M$ is contained in its unique maximal ideal. We can further generalize this containment to irreducible ideals as we see in the following proposition, which also extends the corresponding property of rings (see \cite[Theorem 2.1(ii)]{Azi08}).

\begin{proposition}
Let $M$ be a monoid. For each proper ideal $J$ of $M $, there is a minimal  irreducible ideal
over $J$.
\end{proposition}

\begin{proof}
Suppose $\mathcal{E}:=\{I\in \ir
\mid I\supseteq J\}.$ Since the unique maximal ideal of $M$ is an irreducible ideal, obviously  $\mathcal{E}\neq \emptyset.$ By Zorn's lemma, $\mathcal{E}$ has a minimal element, which is our desired ideal.
\end{proof}

Before we study further properties of strongly irreducible ideals of monoids, let us pause for some examples of them.
Every   prime ideal of a monoid $M$ is  an  irreducible ideal (see  Proposition \ref{ils}(\ref{siii})), and so is the unique maximal ideal of $M$.
In the monoids $(\mathds{N},\cdot, 1)$ and $(\mathds{Z},\cdot, 1)$, the ideals generated by a set of prime numbers are irreducible.
Using the notion of colon ideals, the following two results show how to generate further examples of  irreducible ideals. 

\begin{proposition}\label{icj}
If $I$ is an irreducible ideal and $J$ is an ideal of a monoid $M$, then $(I:J)$ is an irreducible ideal of $M$.
\end{proposition}

\begin{proof} 
Let $K\cap L\subseteq (I:J)$ for some $K, L\in \id.$ This implies $KJ\cap LJ=(K\cap L)J\subseteq I$. From this we have $KJ\subseteq I$ or $LJ\subseteq I,$ since $I$ is strongly irreducible.
\end{proof}
  
\begin{corollary}
Suppose that  $J$, $\{J_{\lambda}\}_{\lambda \in \Lambda}$, $K$ are ideals and  $I$, $\{I_{\omega}\}_{\omega \in \Omega}$ are irreducible ideals of a monoid $M$. Then  $ (I\colon J),$  $((I\colon J)\colon K),$ $(I\colon JK),$ $((I\colon K)\colon J),$ $ (\bigcap_{\omega} I_{\omega}\colon J),$ $ \bigcap_{\omega}(I_{\omega}\colon J),$ $ (I\colon \sum_{\lambda}J_{\lambda}),$ and $ \bigcap_{\lambda}(I\colon J_{\lambda})$ are all irreducible ideals of $M$.
\end{corollary}

Recall that a \emph{monoid homomorphism} $\phi\colon M\to M'$ is a map $\phi$ from $M$ to $M'$ with the property:
\begin{itemize}
\item[]$\bullet$ $\phi(1)=1$,
	
\item[]$\bullet$ $\phi(mm')=\phi(m)\phi(m')$,
\end{itemize} 
for all $m, m'\in M.$ If $J$ is a an ideal of $M'$, then we denote inverse image $\phi\inv (J)$ by $J^c.$ If $I\in \id$, then the ideal generated by $\phi(I)$ is denoted by $I^e.$ The \emph{kernel of} $\phi$ is defined by \[\mathrm{ker}(\phi):=\left\{ (m,m')\in M\times M\mid \phi(m)=\phi(m') \right\}.\]   It is well-known that inverse image of a prime ideal under a ring homomorphism is a prime ideal. However, that fails to hold for strongly irreducible ideals. A sufficient condition for the preservation of strongly irreducible ideals (of rings) under inverse image is given in \cite[Proposition 1.4]{Sch16}. The following proposition generalizes that result to monoids for irreducible ideals.

\begin{proposition}
Let $\phi\colon M \to M '$ be a surjective monoid homomorphism such that
$\mathrm{ker}(\phi)\subseteq \langle x \rangle $ for each $x\notin \mathrm{ker}(\phi)$. If $J$ is an irreducible ideal of $M '$, then $J^c$ is an irreducible ideal of $M $.
\end{proposition}

\begin{proof} 
Let $x,$ $x'\in M\setminus J^c.$ Then $\phi(x),$ $\phi(x')\in M'\setminus J.$ Since $J$ is strongly irreducible, by Proposition \ref{asi}, there exist $y,$ $y'\in M'$ such that \[\phi(x)y=\phi(x')y'\in M'\setminus J.\] Since $\phi$ is surjective, $y=\phi(z)$ and $y'=\phi(z')$ for some $z,$ $z'\in M$. From this, we obtain $xz,$ $x'z' \notin J^c$ and $(xz, x'z')\in \mathrm{ker}(\phi)\subseteq C,$ where $C$ is the smallest congruence containing $\left\{(mx, m'x)\mid m, m'\in M\right\}.$ From this, it follows that $x'z'\notin J^c.$
\end{proof}

Like localization of rings, one can also construct local monoids. Here we briefly recall some essential facts about it, and for further details, we refer to \cite{Flo15} and \cite{AJ84}. Let $M$ be a monoid. A subset $S$ of $M$ is called \emph{multiplicatively closed} if 
\begin{itemize}
\item[]$\bullet$ $1\in S$, and
	
\item[]$\bullet$ $ss'\in S$, whenever $s,$ $s'\in S$.
\end{itemize}
For any $m\in M$ and $s\in S$, define a set
\[M_S:=\{ m/s\mid m/s=m'/s'\;\text{whenever there exists}\;u\in S\;\text{such that}\;(ms')u=(m's)u\}.\] The multiplication on $M_S$ is defined by $(m/s)\cdot (m'/s'):=(mm')(ss')$ and the multiplicative identity of $M_S$ is $1/1$. The system $(M_S, \cdot, 1/1)$ is called the \emph{local monoid} with respect to $S$ or \emph{localization of} $M$ at $S$. The following result from \cite[Proposition 2.4.3(iii)]{Flo15} identifies ideals of a monoid $M$ that are in 1-1 correspondence with ideals of $M_S$.

\begin{proposition}
Let $M$ be a monoid and let $S\subseteq M\setminus \{0\}$ be a multiplicatively closed subset of $M$. Then the proper ideals of $M_S$ correspond to the ideals of $M$ contained in $M\setminus S.$
\end{proposition}

Similar to the above, we can also identify irreducible ideals of $M$ that are in 1-1 correspondence with irreducible ideals of $M_S$, and this will be demonstrated in the next theorem. This result generalizes \cite[Theorem 3.1]{Azi08}.

\begin{theorem}\label{ooc}
Let $M $  be a monoid and let $S$ be a multiplicatively closed subset of $M $. For each $I\in \mathcal{I}(M_S)$, let
$I^c := \{m \in M  \mid m/1 \in I \} = I \cap M$
and let $C := \{I^c \mid I\in \mathcal{I}(M_S)\}.$ Then
there is a 1-1 correspondence between $\mathcal{I}^+(M_S)$
and $\ir$ contained in $C$ which do not meet $S$.
\end{theorem}

\begin{proof}
Let $I\in M_S$. Evidently, $I^c\neq M$. Also, it is clear that $I^c\in C$ and $I^c\cap S=\emptyset$. Suppose $J$, $K\in \id$ such that $J\cap K\subseteq I^c$. This implies
\[J_S\cap K_S=(J\cap K)_S\subseteq I^c_S=I.\]
Since $I$ is strongly irreducible, we must have $J_S\subseteq I$ or $K_S\subseteq I$. This subsequently implies that $J\subseteq J_S^c\subseteq I^c$ or $K\subseteq K_S^c\subseteq I^c$, showing that $I^c\in \ir$. For the converse, suppose that $I\in \ir,$ $I\cap S=\emptyset$, and $I\in C$. It is clear that $I_S\neq M_S$. Suppose $J,$ $K\in M_S$ such that $J\cap K\subseteq I_S$. This implies
\[J^c\cap K^c=(J\cap K)^c\subseteq I_S^c.\]
Since $I\in C$, we must have $I_S^c=I$. Therefore, $J^c\cap K^c\subseteq I$. Since $I$ is strongly irreducible, we have either $J^c\subseteq I$ or $K^c\subseteq I$. Hence, either $J=J_S^c\subseteq I_S$ or $K=K_S^c\subseteq I_S,$ which implies that $I_S\in \mathcal{I}^+(M_S).$
\end{proof}

\begin{corollary}
Suppose $M$ is a monoid and $S$ is multiplicatively closed subset of $M$. If $I \in \ir$ and $I$ is a primary ideal such that $I\cap S=\emptyset$, then $I_S\in \mathcal{I}^+(M_S)$ and a primary ideal of $M_S$.
\end{corollary}

In terms of local monoids, the next result provides two sufficient conditions for primary ideals of a monoid to be irreducible.

\begin{proposition}
Let $M$ be a monoid. Then the following are equivalent.
\begin{enumerate}
\item For the  maximal ideal $\mathfrak{m}$ of $M$, every primary ideal of $M_{\mathfrak{m}}$ is irreducible.

\item Every primary ideal of $M$ is irreducible.

\item For any prime ideal $P$ of $M$, every primary ideal of $M_P$ is irreducible.
\end{enumerate}
\end{proposition}

\begin{proof}
(1)$\Rightarrow$(2): Suppose $I$ is a primary ideal of $M$. Obviously, $I\subseteq \mathfrak{m}.$ Then $I_{\mathfrak{m}}$ is a primary ideal of $M_\mathfrak{m}$, and that by assumption implies $I_{\mathfrak{m}}\in \mathcal{I}^+({M_{\mathfrak{m}}}).$ It follows from Theorem \ref{ooc} that $I_{\mathfrak{m}}^c\in \ir$ and as $I$ is primary, we must have $I_{\mathfrak{m}}^c=I$. This proves that $I\in \ir$.

(2)$\Rightarrow$(3): Suppose that $I$ is a primary ideal of $M_{P}$. Now $I^c$ is a primary ideal of $M$ with the properties: $I^c\cap (M\setminus P)=\emptyset$ and $I^c\in C$. Moreover, by our assumption, $I^c\in \ir$. Therefore, by Theorem \ref{ooc}, we have $I^c_{P}=I\in \mathcal{S}(M_{P})$.

(3)$\Rightarrow$(1): Straightforward.
\end{proof}

The next result gives a representation of an ideal on a monoid in terms of irreducible ideals, and it generalizes \cite[Corollary 2]{Is{e}56}.

\begin{proposition}\label{lfi}
Every ideal $I$ of a monoid $M$ can be represented as follows:
\[I=\bigcap_{\substack{J\supseteq I\\ J\in \ir}}  J.\]
\end{proposition}

\begin{proof}
Let $I\in \id$ and consider the set $\mathcal{E}:=\left\{J\in \ir\mid J\supseteq I\right\}.$ Since $M\in \ir$, the set $\mathcal{E}$ is nonempty. It is evident that $I\subseteq \bigcap_{J\in \mathcal{E}} J.$ To have the other inclusion, observe that it is sufficient to show the claim: if $0\neq x\in M$ and if $I\in \id$ such that $x\notin I$, then there exists a $J\in \ir$ with the property that $J\supseteq I$ and $x\notin J$. Now this fact follows from a simple application of Zorn's lemma.
\end{proof}

Recall that a monoid $M $
is called \emph{Noetherian} if it satisfies the
ascending chain condition on ideals. 
Next, we are interested in Laskerian monoids, where a \emph{Laskerian} monoid is a monoid in which every ideal has a primary decomposition We observe the following result from \cite[Theorem 2.6.2]{Flo15}. 

\begin{theorem}
In a Noetherian monoid, every ideal can
be written as the finite intersection of irreducible primary ideals.
\end{theorem}
This implies that every Noetherian monoid has the primary decomposition property, or equivalently, every Noetherian monoid is Laskerian. 

\begin{proposition}
\label{lask}
If $M $ is a Laskerian monoid, then every irreducible ideal is a primary ideal.	
\end{proposition}

\begin{proof}
Suppose that $I$ is an irreducible ideal of $M$. Since $M$ is Laskerian, there exist $J_i\in \ir$ (where $1\leqslant i\leqslant n$) such that $I=\bigcap_{J\in \ir} J_i.$ In particular, this implies $I\supseteq \bigcap_{J\in \ir} J_i.$ Since $I$ is strongly irreducible, there exists a $j\in \{1, \ldots, n\}$ such that $J_j\subseteq I$. These all together gives
\[J_j\subseteq I=\bigcap_{J\in \ir} J_i\subseteq J_j,\]
implying that $I$ is irreducible, and since $M$ is Noetherian, by \cite[Lemma 2.6.1]{Flo15}, $I$ is a primary ideal.
\end{proof}

\begin{proposition}
Suppose $M$ is a monoid in which every primary ideal is  
irreducible. Then every minimal primary decomposition of an ideal of $M$
is unique.
\end{proposition}

\begin{proof}
Suppose that $I\in \id$ with two minimal primary decomposition representations:
\[I=\bigcap_{i=1}^nP_i=\bigcap_{i=1}^m P_i'.\]
Without loss of generality, assume that $n\leqslant m$. Since $\bigcap_{i=1}^n P_i\subseteq P_1'$ and since $P_1'\in \ir,$ there exists $j\in \{1, \ldots, n\}$ such that $P_j\subseteq P_1'$. Similarly, the facts: $\bigcap_{i=1}^m P'_i\subseteq P_j$ and $P_j\in \ir$ implies  there exists $k\in \{1, \ldots, m\}$ such that 
\[P_k'\subseteq P_j\subseteq P_1'.\]
Since $\bigcap_{i=1}^m P'_i$ is a minimal primary decomposition, we must have $P_k'= P_1'$, and so, $k=1$. Hence, $P_1'=P_j$. Assume that $P_1'=P_1$. Proceeding with induction, we eventually arrive at the conclusions: $P_i'=P_i$ and $n=m$.
\end{proof}

\section*{Declaration } This research has not received any form of funding. Also, the following are not applicable with respect to this work: (a) disclosure of potential conflicts of interest, (b) research involving human participants and/or animals, (c) informed consent.



\begin{thebibliography}{100}


\bibitem{Aub62}
K. E. Aubert, Theory of
$x$-ideals, \textit{Acta Math.}, \textbf{107} (1962), 1--52.
  	
\bibitem{Azi08} A. Azizi, Strongly irreducible ideals, \emph{J. Aust. Math. Soc.}, \textbf{84} (2008), 145--154.

\bibitem{AA08}
R. E. Atani and S. E. Atani, Ideal theory in commutative semirings, \textit{Bul. Acad.
Stiin te Repub. Mold. Mat.}, (2008), 14--23.

\bibitem{AJ84} D. D. Anderson, and E. W. Johnson, Ideal theory in commutative semigroups, \textit{Semigroup forum}, \textbf{30}(2) (1984), 127--158. 
	

\bibitem{Flo15}
J. Flores, \textit{Homological algebra for commutative monoids},
Thesis (Ph.D.), The state university of New Jersey, 2015. 

\bibitem{FGKT17}
Y. Fan, A. Geroldinger, F. Kainrath,  and S. Tringali,  Arithmetic of commutative semigroups with a focus on semigroups of ideals and modules,
\textit{J. Algebra Appl.}, \textbf{16}(12) (2017), 1750234, 42 pp.

\bibitem{G99} 
J. S. Golan, \textit{Semirings and their applications}, Springer, 1999.

\bibitem{GK22}
A. Geroldinger, and M. A. Khadam, On the arithmetic of monoids of ideals,
\textit{Ark. Mat.}, \textbf{60}(1) (2022), 67--106.

\bibitem{GGT21}
A. Geroldinger, F. Gotti, and S. Tringali, On strongly primary monoids, with a focus on Puiseux monoids, \textit{J. Algebra}, \textbf{567} (2021), 310--345.

\bibitem{GR20}
A. Geroldinger and Roitman, On strongly primary monoids and domains,
\textit{Comm. Algebra} \textbf{48}(9) (2020), 4085--4099.

\bibitem{GZ20}
A. Geroldinger and Q. Zhong, Factorization theory in commutative monoids,
\textit{Semigroup Forum}, \textbf{100}(1) (2020), 22--51.

\bibitem{GHL07}
A. Geroldinger, W. Hassler, and G. Lettl, On the arithmetic of strongly primary monoids,
\textit{Semigroup Forum}, \textbf{75}(3) (2007), 568--588.

\bibitem{Ger96}
A. Geroldinger, On the structure and arithmetic of finitely primary monoids,
\textit{Czechoslovak Math. J.}, \textbf{46}(4) (1996), 677--695.

\bibitem{CHWW15}
G. Corti\~{n}as,  C. Haesemeyer,  M. E. Walker, and C. Weibel, Toric varieties, monoid schemes and \textit{cdh} descent,
\textit{J. Reine Angew. Math.}, \textbf{698} (2015), 1--54.

\bibitem{H-K98}
F. Halter-Koch, \textit{Ideal systems. An introduction to multiplicative ideal theory}, Marcel Dekker, 1998.

\bibitem{Is{e}56}
K. Is\'{e}ki, Ideal theory of semiring , \textit{Proc. Japan. Acad.}, \textbf{32} (1956), 554--559.

\bibitem{Kat94}
K. Kato, Toric singularities, \textit{Amer. J. Math.}, \textbf{116}(5) (1994),
1073--1099.

\bibitem{Ogu18}
A. Ogus, \textit{Lectures on logarithmic algebraic geometry}, Cambridge University Press, 2018.

\bibitem{Put88}
M. S. Putcha, \textit{Linear algebraic monoids}, Cambridge University Press, 1988.

\bibitem{Ren05}
L. E. Renner, \textit{Linear algebraic monoids}, Springer, 2005. 

\bibitem{Sch16} N. Schwartz, Strongly irreducible ideals and truncated valuations, \emph{Comm. Algebra}, \textbf{44}(3) (2016), 1055--1087. 
\end{thebibliography}
\end{document}